\documentclass[reqno, 12pt]{amsart}
\pdfoutput=1
\makeatletter
\let\origsection=\section \def\section{\@ifstar{\origsection*}{\mysection}} 
\def\mysection{\@startsection{section}{1}\z@{.7\linespacing\@plus\linespacing}{.5\linespacing}{\normalfont\scshape\centering\S}}
\makeatother        

\usepackage{amsmath,amssymb,amsthm}
\usepackage{mathrsfs}
\usepackage{mathabx}\changenotsign
\usepackage{bbm}
 
\usepackage{xcolor}
\usepackage[backref]{hyperref}
\hypersetup{
    colorlinks,
    linkcolor={red!60!black},
    citecolor={green!60!black},
    urlcolor={blue!60!black}
}

\usepackage[open,openlevel=2,atend]{bookmark}

\usepackage[abbrev,msc-links,backrefs]{amsrefs} 
\usepackage{doi}

\renewcommand{\PrintDOI}[1]{\doi{#1}}

\usepackage[T1]{fontenc}
\usepackage{lmodern}
\usepackage[babel]{microtype}
\usepackage[english]{babel}

\usepackage{geometry}
\numberwithin{equation}{section}

\usepackage{enumitem}
\def\rmlabel{\upshape({\itshape \roman*\,})}

\def\alabel{\upshape({\itshape \alph*\,})}

\let\polishlcross=\l
\def\l{\ifmmode\ell\else\polishlcross\fi}

\let\emptyset=\varnothing
\let\setminus=\smallsetminus

\makeatletter
\def\moverlay{\mathpalette\mov@rlay}
\def\mov@rlay#1#2{\leavevmode\vtop{   \baselineskip\z@skip \lineskiplimit-\maxdimen
   \ialign{\hfil$\m@th#1##$\hfil\cr#2\crcr}}}
\newcommand{\charfusion}[3][\mathord]{
    #1{\ifx#1\mathop\vphantom{#2}\fi
        \mathpalette\mov@rlay{#2\cr#3}
      }
    \ifx#1\mathop\expandafter\displaylimits\fi}
\makeatother

\newcommand{\dcup}{\charfusion[\mathbin]{\cup}{\cdot}}

\DeclareFontFamily{U}  {MnSymbolC}{}
\DeclareSymbolFont{MnSyC}         {U}  {MnSymbolC}{m}{n}
\DeclareFontShape{U}{MnSymbolC}{m}{n}{
    <-6>  MnSymbolC5
   <6-7>  MnSymbolC6
   <7-8>  MnSymbolC7
   <8-9>  MnSymbolC8
   <9-10> MnSymbolC9
  <10-12> MnSymbolC10
  <12->   MnSymbolC12}{}
\DeclareMathSymbol{\powerset}{\mathord}{MnSyC}{180}

\let\epsilon=\varepsilon

\let\rho=\varrho
\let\theta=\vartheta
\let\kappa=\varkappa

\newcommand{\cE}{\mathcal{E}}
\newcommand{\cF}{\mathcal{F}}
\newcommand{\cG}{\mathcal{G}}
\newcommand{\cH}{\mathcal{H}}

\newcommand{\ad}{\mathrm{ad}}

\newcommand{\EX}{\mathfrak{E}}

\theoremstyle{plain}
\newtheorem{theorem}{Theorem}[section]
\newtheorem{fact}[theorem]{Fact}
\newtheorem{prop}[theorem]{Proposition}
\newtheorem{claim}[theorem]{Claim}

\newtheorem{lemma}[theorem]{Lemma}

\theoremstyle{definition}

\newtheorem{definition}[theorem]{Definition}

\usepackage{accents}

\let\phi=\varphi
\def\red{\text{red}}

\begin{document}
\title[Stars and cliques]
{A tale of stars and cliques}

\author[Tomasz {\L}uczak]{Tomasz {\L}uczak} 
\address{Adam Mickiewicz University,
Faculty of Mathematics and Computer Science, Pozna\'n, Poland}
\email{tomasz@amu.edu.pl}
\email{joaska@amu.edu.pl}
\thanks{The first author was partially 
supported by NCN grant 2012/06/A/ST1/00261.}

\author{Joanna Polcyn}

\author[Christian Reiher]{Christian Reiher}
\address{Fachbereich Mathematik, Universit\"at Hamburg, Hamburg, Germany}
\email{Christian.Reiher@uni-hamburg.de}


\newcommand{\irt}{\mathcal{J}(r,t)}
\newcommand{\irtt}{r,t}
\newcommand{\el}{{rt}}
\newcommand{\ha}{c_{4}}
\newcommand{\hb}{c_{7}}
\newcommand{\hc}{c_{5}}
\newcommand{\hn}{\hat n}
\newcommand{\hcc}{c_6}
\newcommand{\kr}{\operatorname{cr}_e}
\newcommand{\krr}{\operatorname{cr}_v}
\newcommand{\pp}{P}
\newcommand{\ppkl}{P^{k}_\ell}
\newcommand{\pcz}{P^{4}_2}
\newcommand{\F}{\mathcal{F}}
\newcommand{\G}{\mathcal{G}}
\newcommand{\tH}{\overline{H}}
\newcommand{\hH}{\widehat{H}}
\newcommand{\hG}{\widehat{G}}
\newcommand{\hR}{\widehat{R}}
\newcommand{\hU}{\widehat{U}}
\newcommand{\bl}{\bar{\ell}}
\newcommand{\bG}{\bar{G}}
\newcommand{\sg}{\operatorname{sg}}
\newcommand{\ff}{f}
\let\vn=\varnothing
\let\sm=\setminus

\hyphenation{e-li-mi-nate essen-tia-lly corres-pon-ding}

\begin{abstract}
We show that for infinitely many natural numbers $k$ there are $k$-uniform hypergraphs 
which admit a `rescaling phenomenon' as described in~\cite{LP}. More precisely, 
let~$\mathcal{A}(k,I, n)$ denote the class of $k$-graphs on $n$ vertices 
in which the sizes of all pairwise intersections of edges belong to a set~$I$. 
We show that if $k=rt^2$ for some $r\ge 1$ and~$t\ge 2$, and~$I$ is chosen 
in some special way, the densest graphs in $\mathcal{A}(rt^2,I, n)$ are either 
dominated by stars of large degree, or basically, they are `$t$-thick' $rt^2$-graphs 
in which vertices are partitioned into groups of $t$ vertices each and every edge 
is a union of $tr$ such groups. 
It is easy to see that, unlike in stars, 
the maximum degree of $t$-thick graphs is 
of a lower order than the  number of its edges. Thus,  if we study the 
graphs from $\mathcal{A}(rt^2,I, n)$ with a prescribed number of edges $m$ which minimise the 
maximum degree, around the value of $m$ which is the number of edges of the largest 
$t$-thick graph,
a rapid, discontinuous phase transition can be observed. Interestingly, these two types 
of $k$-graphs determine the structure of all hypergraphs in $\mathcal{A}(rt^2,I, n)$. 
Namely, we show that each such hypergraph can be decomposed 
into a $t$-thick graph $H_T$, a special collection $H_S$ of stars, 
and a sparse `left-over' graph $H_R$. 
\end{abstract}

\maketitle

\section{Introduction}
By a {\it set system} we mean a pair ${S=(V, \cE)}$ 
such that $\cE$ is a collection of subsets of $V$. The members of $V$ are usually 
referred to as the {\it vertices} of the set system, whilst the members of $\cE$
are called {\it edges}. If all members of $\cE$ are of the same cardinality $k\ge 0$ 
we call $S$ a {\it $k$-uniform hypergraph} or, more brief\-ly, a {\it $k$-graph}.

Occasionally we identify a hypergraph $H$ with its set of edges, denoting, for example, 
by~$|H|$ the number of edges in $H$. For a given set $I$ of nonnegative integers, we say that 
a $k$-graph~$H$ is \emph{$I$-intersecting} if $|e\cap f|\in I$ holds for all
$e, f\in H$. Starting with the seminal work~\cite{EKR61} of Erd\H{o}s, Ko, and Rado,
the study of $I$-intersecting hypergraphs and set systems has a long tradition in extremal 
combinatorics
(see, e.g., \cites{AK97, DEF76, BF80, W84, FF85, FR87, MR14} for some milestones). 
Let us remark that 
sometimes in the literature (e.g., \cites{DEF76, BF80, FF85}) an $I$-intersecting 
$k$-graph on $n$ vertices  is called an $(n,k, I)$-system.

Motivated by the stability of extremal hypergraphs for the $3$-uniform loose path of 
length~$3$ the first two authors studied $\{0, 2,3,4\}$-intersecting $4$-graphs in~\cite{LP}. 
The aim of the present article is to extend their results to the more general family $\irt$
which consists of all $I$-intersecting $rt^2$-graphs, where 
$r\ge 1$ and $t\ge 2$ are arbitrary integers and 
\[
	I=\bigl\{s\colon t\mid s\textrm{\  or\ }s\ge rt(t-1)\bigr\}\,.
\]

This choice of the set of permissible intersections may look bizarre at first
and our main incentive to study it came from the aesthetical merits of the results
we hoped to obtain: to explain those, we start from the observation that there are 
two quite different examples of dense  $rt^2$-graphs $H\in \irt$ on $n$ vertices  
with~$\Theta(n^{rt})$ edges.

The most obvious one  is the {\it full $\bigl(rt(t-1)\bigr)$-star}, i.e., a hypergraph $H$ 
with a distinguished 
$rt(t-1)$-set~$S$ of vertices, called the {\it centre} of the star, such that the edges of $H$
are precisely the $rt^2$-supersets of~$S$. Clearly such a star has exactly 
$\binom{n-rt^2+rt}{rt}$ edges and it can be shown that, for large $n$, it is the unique 
hypergraph which maximises the number of edges among all hypergraphs in $\irt$ 
on $n$ vertices  (see Proposition~\ref{prop:silly} below).

However, there exists  another natural construction of dense  $rt^2$-graphs $H\in \irt$
with~$n$ vertices and $\Theta(n^{rt})$ edges. 
It proceeds by splitting the vertex set into 
$\lfloor n/t\rfloor$ subsets of size $t$ called {\it teams} (and a small 
number of left-over vertices) and to declare an $rt^2$-set to be an edge if and 
only if it is a union of~$rt$ teams. 
We call the resulting hypergraph a {\it thick clique} and to its subhypergraphs we refer as 
{\it thick hypergraphs}. 
Note that each thick hypergraph  has the property 
that for any two edges $e$ and $f$ the number $|e\cap f|$ is a multiple of $t$ and, 
hence, it indeed belongs to $\irt$.

The point that interests us here is that even though both the star and the thick clique 
have $\Theta(n^{rt})$ edges, their maximum vertex degrees are of different orders of magnitude.
In fact, while the vertices belonging to the centre of a star have degree $\Omega(n^{rt})$,
the maximum degree of a thick clique is easily seen to be only $O(n^{rt-1})$. Perhaps
surprisingly, it turns out that this phenomenon arises in a very ``discontinuous''
manner: As soon as a graph from $\irt$ has one edge more than the thick
clique, it needs to contain a vertex of degree~$\Omega(n^{rt})$. 

This is the main result of the present work which, crudely, can be stated as follows 
(for further structural results see 
Theorems~\ref{thm:dt} and~\ref{prop:25} below).

\begin{theorem}\label{thm:mainr2}
For  $r\ge 1$ and $t\ge 2$ there exists an $n_0$ such that for every $rt^2$-graph 
	$H\in \irt$ with $n\ge n_0$ vertices and at least $\binom{\lfloor n/t\rfloor}{rt}+1$ edges 
	we have $\Delta(H)\ge e(H)/(3t)$.
	
	On the other hand, for every $n\ge rt^2$ a thick clique  $H_0\in \irt$ on $n$ vertices
has~$\binom{\lfloor n/t\rfloor}{rt}$ edges
and
	$\Delta(H_0)=\binom{\lfloor n/t\rfloor -1}{rt-1}$.
\end{theorem}

The main step in the proof of Theorem~\ref{thm:mainr2} is a somewhat surprising 
structural result (see Theorem~\ref{thm:dt} below). It turns out that stars and 
thick hypergraphs which naturally emerge when we study hypergraphs in $\irt$ 
whose density is close to the maximum density $\Theta(n^{rt})$, are natural building 
blocks for all `not too sparse' members $\irt$. More specifically,
we show that up to an `error of lower order,' i.e., up to at most $O(n^{rt-1})$ 
edges, any such hypergraph arises by attaching ``non-overlapping'' stars to a thick hypergraph.

This article is organised as follows. 
In the next section we state a precise version of the structure theorem mentioned above. 
Then, in Section~\ref{sec:new}, we show that it does indeed imply
Theorem~\ref{thm:mainr2} and give more structural characterisations of 
dense hypergraphs in $\irt$ with small maximum degree.
Section~\ref{sec:tools} collects some tools
needed for the proof of this structure theorem including
a `decomposition lemma' (see Lemma~\ref{lem:dl} below) that might 
have some other applications as well. 
Finally, in Section~\ref{sec:4}, we prove the structure theorem. 

\section{The structure theorem}\label{sec:mmd}

We begin this section with some definitions allowing us to formulate a 
precise version of the structure theorem for  $rt^2$-graphs $H\in \irt$. 

Let $k\ge s\ge 0$ be integers. A $k$-uniform hypergraph $H=(V, E)$ 
with a set $S\subseteq V$ of distinguished vertices of size $|S|=s$ is an~\emph{$s$-star} 
if $S\subseteq e$ holds for all edges $e\in E$. We call~$S$ the \emph{centre} of the star and 
$\bigcup_{h\in E}(h\setminus S)$ is referred to as the \emph{body} of the star. 
A collection of stars is said to be \emph{semi-disjoint} if their centres are distinct 
and their bodies are mutually disjoint.
Of course, an $s$-star $H$ on $|V|=n$ vertices can have at most $\binom{n-s}{k-s}$ edges. 
If this happens we say that $H$ is a \emph{full $s$-star} and denote it by~$S^k_{n,s}$. 

Next, for a given hypergraph $H=(V,E)$, we say that a subset $W\subseteq V$ of its vertex set 
is \emph{inseparable} in $H$, if for all edges $h\in E$ we have $W\cap h\in \{\emptyset,W\}$.
Now consider three natural numbers $k$, $t$, and $n$ satisfying $t\mid k$,
and suppose that a set $V$ of $n$ vertices is partitioned into $\lfloor n/t\rfloor$
many $t$-subsets called {\it teams} and fewer than $t$ further vertices. 
By~$\widetilde{K}^{k}_{n,t}$ we denote the {\em thick $(k, n, t)$-clique}, 
i.e., the $k$-graph on $n$ vertices whose $\binom{\lfloor n/t\rfloor}{k/t}$ 
edges are all possible unions of some $k/t$ of these teams. We refer to its subhypergraphs as 
{\it $t$-thick} or just {\it thick hypergraphs}. 
Evidently the teams
are inseparable in $\widetilde{K}^{k}_{n,t}$ and a $k$-graph $H$ on $n$ vertices 
possessing~$\lfloor  n/t\rfloor$ 
mutually disjoint inseparable $t$-sets of vertices is a subhypergraph of the thick 
clique~$\widetilde K^{k}_{n,t}$.  
 
Finally, for positive integers $t$, $\ell$, and $a$, we define  a class
$\F(t,\ell,a)$ of $\ell t$-graphs as follows.

\begin{definition}\label{dfn:family}
	For given natural numbers $t,\ell$, and $a$, we say that an $\ell t$-graph $H=(V,E)$ 
	belongs to the class $\F(t,\ell,a)$ if  
	there exist partitions 
		\[
		V=V_T\cup V_S\cup V_R \quad \text{and} \quad H=H_T\cup H_S\cup H_R\,,
	\]
		such that 
	
	\begin{enumerate}[label=\rmlabel]
		\item\label{it:VR} $V_T$ is a union of  inseparable $t$-subsets of $V$, 
			and $H_T=H[V_T]$; 
		\item\label{it:VS} $H_S= \{h\in H\colon |h\cap V_S|=\ell\}$ 
			consists of semi-disjoint 
			$\bigl(\ell(t-1)\bigr)$-stars with their centres in $V_T\cup V_R$ 
				and their bodies in $V_S$;		
		\item\label{it:VSS} any edge of $H$ that intersects the body of a star 
			$S^*\subseteq H_S$ contains the centre of $S^*$; 
		\item\label{it:VT} 	$|H_R|\le |V_T||V_S|n^{\ell-3}+|V_R|an^{\ell-2}$.
	\end{enumerate}
\end{definition}

Now the structure theorem for  $rt^2$-graphs $H\in\irt$ promised in the introduction 
can be stated as follows.

\begin{theorem}[Structure Theorem]\label{thm:dt}
	For all integers $r\ge 1$ and $t\ge 2$ we have 
	\[
		\irt\subseteq \F(t,rt,(rt^2)^{r^3t^6})\,.
	\]
\end{theorem}

The proof of this result is deferred to Section~\ref{sec:4}. We conclude this 
section by pointing out that the structure theorem quickly allows us to determine the 
extremal $rt^2$-graphs in~$\irt$. The following 
statement shows that the extremal hypergraph, the full $\bigl(\ell(t-1)\bigr)$-star, 
is unique and stable for this problem. 

\begin{prop}\label{prop:silly}
	Given natural numbers $\ell\ge t\ge 2$, $a$, and $c$, there exists an integer~$n_*$ such that 
	every $\ell t$-graph $H=(V, E)\in \F(t,\ell,a)$ with $|V|=n\ge n_*$ and 
	${e(H)\ge \binom{n-c}{\ell}-\frac{n^{\ell-1}}{2(\ell-1)!}}$ 
	edges is obtained from an $\bigl(\ell(t-1)\bigr)$-star by adding at most 
	$\binom{c}{\ell t}$ further edges. 
	
	Moreover, if $e(H)\ge \binom{n-\ell t}{\ell}-\frac{n^{\ell-1}}{2(\ell-1)!}$, then $H$ 
	is an $\bigl(\ell(t-1)\bigr)$-star. 
	
	In particular, each $H\in \F(t,\ell,a)$ has at most $\binom{n-\ell(t-1)}{\ell}$ 
	edges, and this maximum is achieved only if $H$ is isomorphic to the full 
	$\bigl(\ell(t-1)\bigr)$-star $S^{\ell t}_{n, \ell(t-1)}$.
\end{prop} 

\begin{proof} Let us choose $n_*$ and $D$ with 
	$n_*\gg D\gg \max(t, \ell, a)+c$ such that all inequalities below  hold for $n\ge n_*$.
	Moreover, let $H=(V, E)$ with 
		\[
		|V|=n\ge n_* \quad \text{and} \quad 
		|E|=m\ge \binom{n-c}{\ell}-\frac{n^{\ell-1}}{2(\ell-1)!}
	\]
		be a $\ell t$-graph from $\F(t,\ell,a)$, and take partitions 
		\[
		V=V_T\cup V_S\cup V_R
		\quad \text{ as well as } \quad 
		H=H_T\cup H_S\cup H_R
	\]
		exemplifying this.

	Clearly 
		\begin{equation*}%\label{eq:m1}
	 	m\le\binom{|V_T|/t}{\ell}+\binom{|V_S|}{\ell}+|H_R|\,,
	\end{equation*}
	where
		\[
		|H_R|\le |V_T||V_S|n^{\ell-3}+|V_R|an^{\ell-2}\le 
		(|V_T|+|V_R|)an^{\ell-2}\le an^{\ell-1}\,, 
	\]
		and so
		\begin{equation}\label{eq:m2}
		m\le\binom{|V_S|+|V_T|/t}{\ell}+an^{\ell-1}\,.
	\end{equation}
		Assume first that $|V_S|< n-2D$. Then 
		\[
		|V_S|+\frac 1t|V_T|\le |V_S|+\frac 12(n-|V_S|)=\frac 12(n+|V_S|)\le n-D
	\]
		and so, by \eqref{eq:m2},
		\[
		m\le\binom{n-D}{\ell}+an^{\ell-1}
		<\binom{n-c}{\ell}-\frac{n^{\ell-1}}{2(\ell-1)!}\,.
	\]
		As this contradicts our assumption, we may conclude that $|V_S|\ge n-2D$. 
	
	In particular, we have $|V_T|+|V_R|\le 2D$, and $|H_T\cup H_R|\le (2Da+1)n^{\ell-2}$.
	Consider the largest star $S^*=(V^*,E^*)$ in $H_S$ and let $s$ be its centre. 
	Then $|V^*|\ge n/2$, since otherwise
	\[
		m=|H_S|+|H_T|+|H_R|\le \binom{n/2}{\ell} + \binom{n/2}{\ell} + (2Da+1)n^{\ell-2}
		<\binom{n-c}{\ell}-\frac{n^{\ell-1}}{2(\ell-1)!}\,.
	\]
	Now assuming  $|V^*\cap V_S|\le n-c-1$ we could argue that
	\[
		m\le \binom{n-c-1}{\ell}+\binom{c+1}{\ell} + (2Da+1)n^{\ell-2}
		<\binom{n-c}{\ell}-\frac{n^{\ell-1}}{2(\ell-1)!}\,,
	\]
	which, again, contradicts our assumption on $m$. 

	This proves that $|V^* \cap V_S|\ge n-c$. By Definition~\ref{dfn:family}\ref{it:VSS}, 
	all edges of $H$ intersecting $V^* \cap V_S$ contain $s$ and therefore they form an 
	$\bigl(\ell(t-1)\bigr)$-star. Since $|V\setminus (V^*\cap V_S)|\le c$, there can 
	be at most $\binom{c}{\ell t}$ edges not belonging to this star, which establishes
	our first assertion. The moreover-part follows from the observation that in 
	case $c=\ell t$ the only potential further edge, $V\setminus (V^*\cap V_S)$,
	would still contain $s$ and could thus be adjoined to the star.
\end{proof}

\section{Minimum maximum degree} 
\label{sec:new}

Let us first start with the proof of Theorem~\ref{thm:mainr2} which, let us recall, 
states that in each hypergraph from $\irt$ with  $m> \binom{\lfloor n/t\rfloor}{rt}$ 
edges there exists a big star which contains a positive fraction of all edges; 
moreover  thick cliques show that this result is sharp. We prove this result in a 
slightly stronger form, which gives a better estimate for the size of the biggest star 
for dense graphs.
Besides, it states that each graph from $\irt$  which has 
nearly~$\binom{\lfloor n/t\rfloor}{rt}$ edges and small maximum degree is thick. 
Here, for $H\in\irt$, by $H_S$ we denote a subgraph consisting of 
$\bigl(rt(t-1)\bigr)$-stars as obtained by 
applying the Structure Theorem~\ref{thm:dt} to $H$.

\begin{theorem}\label{thm:mainr3}
	For $r\ge 1$ and $t\ge 2$ there exists an $n_0$ such that, for every $rt^2$-graph 
	$H\in \irt$ with $n\ge n_0$ vertices and $m\ge \binom{\lfloor n/t\rfloor}{rt}+1$ 
	edges, $H_S$ contains an $\bigl(rt(t-1)\bigr)$-star with at least $\hn$ vertices 
	in the body and at least $m\hn/n-n^{rt-1}> m/(3t)$ edges, where
	\begin{equation}\label{eq:hn}
		\hn=\hn(n,m)=\min\Big\{N:\binom{N-1}{rt-1}\ge \frac{rtm}{n}- rtn^{rt-2}\Big\}
		\ge\frac{n}{t^{rt/(rt-1)}}-(3rt)^{3rt}>\frac{2n}{5t}\,.
	\end{equation}
	On the other hand, if
			$\binom{\lfloor n/t\rfloor}{ \el}-\frac {n^{\el-1}}{2t^{\el-1}(\el-1)!}
			\le m\le \binom{\lfloor n/t\rfloor}{ \el}$,
	then each hypergraph $H\in \irt$ with $m$ edges and  $\Delta(H)\le m/(3t)$ is   
	a subgraph of a thick clique $\widetilde{K}^{rt^2}_{n,t}$; in particular,  
	$\Delta(H)\le\binom{\lfloor n/t\rfloor -1}{rt-1}$.
\end{theorem}

\begin{proof}
	For given integers $t\ge 2$, and $r\ge 1$, choose $n_0$ so large 
	that all inequalities below hold for $n\ge n_0$.  
	Moreover, let $H\in \irt$, where the number of edges $m$ satisfies
		\begin{equation}\label{m}
		m\ge \binom{\lfloor n/t\rfloor}{\el}-\frac {n^{\el-1}}{2t^{\el-1}(\el-1)!}
		> \frac{n^{\el}}{t^{\el}(\el) !}-\frac{2rt^2 n^{\el-1}}{t^{\el}(\el-1)!}
		> \left(\frac n{rt^2}\right)^{\el}\,.
	\end{equation}
		
	By Theorem \ref{thm:dt}, $H\in \F(t,rt,(rt^2)^{r^3t^6})$, so let us take partitions 
	$V=V_T\cup V_S\cup V_R$ and $H=H_T\cup H_S \cup H_R$ exemplifying this. 
	Then
	\begin{equation}\label{mp4}
		|H|=|H_T| + |H_S| +|H_R|
		\le \binom{ |V_T|/t}{\el} +\binom{|V_S|}{\el}+|V_T||V_S|n^{\el-3} + |V_R|Cn^{\el-2}\,,
	\end{equation}
	where $C=(rt^2)^{r^3t^6}$. As a straightforward consequence of the above inequality 
	we get the following claim. 
	
	\begin{claim}\label{cl42}
		If $V_S=\emptyset$, then $H_R=\varnothing$ and $m\le \binom{\lfloor n/t\rfloor}{\el}$. 
	\end{claim}
	
	\begin{proof}
		Since $V_S=\emptyset$, the vertex set of $H$ is partitioned into sets $V_T$ and $V_R$, 
		where $|V_T|$ is divisible by $t$ and $|V_R|=n-|V_T|$. Recall that $V_T$ consists of 
		$t$-tuples that are inseparable in $H$. Therefore, if $H_R\neq \emptyset$ then 
		$|V_R|\ge t$ and consequently, by (\ref{mp4}),
		\begin{align*}
			m&=|H_T|+|H_R|\le \binom {(n-|V_R|)/t}{\el}+0+|V_R|Cn^{\el-2} 
			\le\binom {\lfloor n/t\rfloor-1 }{ \el}+2tCn^{\el-2}\\
			&< \binom{\lfloor n/t\rfloor}{ \el}-\frac {n^{\el-1}}{2t^{\el-1}(\el-1)!}\,,
					\end{align*}
		contrary to~\eqref{m}. 
		
		Thus we must have $H_R=\emptyset$ and, hence,  
				\[
		m=|H_T|  		\le \binom{\lfloor n/t\rfloor}{\el}\,. \qedhere
		\]
					\end{proof}
	
It turns out that if $V_S\neq \emptyset$, then the maximum degree must be large.

	\begin{claim}\label{tp4}
		If $V_S\neq \emptyset$, then $H_S$ contains an $\bigl(\el(t-1)\bigr)$-star with 
		at least $\hn$ vertices in the body
		and at least $m\hn/n-n^{rt-1}$ edges, where $\hn$ is 
		defined as in (\ref{eq:hn}).
	\end{claim}

	\begin{proof}		We start with bounding from below the average degree $\ad(G)$ 
		of the $\el$-graph $G_S=(V_S,E_S)$ with the set of vertices $V_S$ and the set 
		of edges $E_S=\{h\cap V_S: h\in H_S\}$. 
		Using the upper bound on $|H_R|$ we get  
		\begin{align*}
			\ad(G_S)
			&=\frac{\el|H_S|}{|V_S|}=\frac{\el(m-|H_R|-|H_T|)}{n-|V_R|-|V_T|}\\
			&\ge \frac{\el m-\el|V_R|Cn^{\el-2}-\el\binom{|V_T|/t}{\el}}{n-|V_R|-|V_T|}
				-\frac{\el |V_T||V_S|n^{\el-3}}{|V_S|}\\
		&=\frac{\el m}{n} + \frac {\frac{\el m}n(|V_R|+|V_T|)-\el|V_R|Cn^{\el-2}
			-\el\binom{|V_T|/t}{\el}}{n-|V_R|-|V_T|}-\el|V_T|n^{\el-3}\,.													\end{align*}
	Here the numerator of the second fraction is, due to~\eqref{m}, at least 
		\[
		\el |V_R|n^{\el-2}\left(\frac{ n}{(rt^2)^{\el}}- C\right)
		+ \el |V_T|\left(\frac 1n\binom{\lfloor n/t\rfloor}{\el}-\frac{n^{\el-2}}{2t^{\el-1}
		(\el-1)!}-\frac 1{|V_T|}\binom{|V_T|/t}{\el}\right)
	\]
		and because of $|V_T|\le n-t$ and the fact that  $n$ is large this term is positive.
	For these reasons we have 
		\[
		\ad(G_S)\ge \frac{\el m}{n} - \el n^{\el-2} > \binom{\hn-2}{\el-1}\,.
	\]
		
	Now, each vertex $v$ of degree at least $\ad(G_S)$ must be contained in a 
	component with at least $\hn$ vertices and, since $G_S$ must contain a component 
	whose average degree is at least~$\ad(G_S)$, 
	each such component must have at least $\ad(G_S)\hn/(\el)$ edges. 
	\end{proof}
	
Finally, we can complete the proof of Theorem~\ref{thm:mainr3}. 
If $m> \binom{\lfloor n/t\rfloor}{\el}$,  then, by Claim~\ref{cl42}, 
we have $V_S\neq \emptyset$, and the first part of Theorem~\ref{thm:mainr3} follows
directly from Claim~\ref{tp4}.
On the other hand, if $H\in \irt$ has $m$ edges, where  
\[
	\binom{\lfloor n/t\rfloor}{\el}-\frac {n^{\el-1}}{2t^{\el-1}(\el-1)!}
	\le m \le \binom{\lfloor n/t\rfloor}{\el} \,, 
\]
and $\Delta(H)\le m/(3t)$, then Claim~\ref{tp4} implies that $V_S$ is empty and, 
by Claim~\ref{cl42}, $H_R$ is empty as well. Thus, $H$  must be a subgraph of a thick  
$(rt^2,n,t)$-clique $\tilde K^{rt^2}_{n,t}$.\end{proof}

Once we know that dense graphs from $\irt$ with $m>\binom{\lfloor n/t\rfloor}{\el} $ 
contain vertices of large degree one may ask 
about the structure of graphs which, for a given $m=m(n)$, minimise the maximum degree.
A natural conjecture is that they can be expressed
as  a union of large disjoint $\bigl(rt(t-1)\bigr)$-stars with,
perhaps, some limited number of extra edges like those which intersect the centres of these
stars in sets whose sizes are multiples of $t$. 

In \cite{LP} such a result is proved for the family $\mathcal{J}(1,2)$. 
Namely, it is shown that, for 
large enough $n$, from each $\{0,2,3,4\}$-intersecting 
$4$-graph with $n$ vertices and $m> \binom{\lfloor n/2\rfloor}{2}$
 edges
  that minimises the maximum degree one can 
remove at most~$128$ edges to get a $4$-graph which consists of at most four $2$-stars and, perhaps, 
some number of isolated vertices 
%The number $4$  here is optimal, while $128$ is clearly not and most likely it can be replaced 
%either by $11$, if we want each of the stars to be a $2$-star, or
%just by $8$, if we are satisfied that each of the four stars is just a $1$-star 
(for details and discussions of this result see~\cite{LP}). 

The remaining part of this section is devoted to the proof of an analogous result 
for~$\irt$ in the general case. As we will see shortly,
a similar result holds whenever $r=1$, while for~$r\ge 2$ a weaker yet quite satisfactory
characterisation of the extremal graphs can be shown. Nevertheless, in order to state our theorem 
more precisely, we need some notation, analogous to those used in \cite{LP}.
     
We define the 
{\it minimum maximum-degree function} of $\irt$ by setting
\[
	\ff(\irtt; n, m)=\min\bigl\{\Delta(H): H=(V,E)\in\irt, |V|=n, \text{\ and\ } |E|=m \bigr\}
\]
for all nonnegative integers $n$ and $m$. The corresponding collection of extremal hypergraphs
is denoted by $\EX(\irtt; n, m)$. 

Note that the function $\ff(\irtt;n,m)$ is always bounded from
below by the average degree $rt^2 m/n$; on  the other hand, one can always find a thick graph 
from $\irt$ such that the degrees of all vertices, except at most $t-1$, are within distance one 
from each other. Hence, from Theorem~\ref{thm:mainr3} it follows that whenever 
$m\le \binom{\lfloor n/t\rfloor}{\el}$ and $n$ is large enough we have 
%\[
%	\lceil rt^2m/n\rceil \le \ff(\irtt; n, m)\le  %\lceil rt^2m/n\rceil+1\,,
%\]

$$
\lceil rt^2m/n\rceil \le \ff(\irtt; n, m)\le  \lceil rtm/\lfloor n/t\rfloor\rceil=
\lceil rt^2m/n\rceil(1+O(t/n))\,,
$$
i.e.,\ in this range of $m$ the function $\ff(\irtt; n, m)$ is determined up to the first order term. 
Thus, it remains to study the value of $\ff(\irtt; n, m)$ and the structure of 
the extremal graphs from  $\EX(\irtt; n, m)$ for $m>\binom{\lfloor n/t\rfloor}{\el}$. 
For this we require one more concept.

Let us say that an $\bigl(rt(t-1)\bigr)$-star $S$ with some number $N$ of vertices in its
body is {\it heavy} if its minimum vertex degree is at least $2r^2t^3N^{rt-2}$. 
Consider the  process when we  repeatedly remove from a star $S$ the (lexicographically first) vertex of smallest degree until the resulting star, possibly empty, is heavy. 
 The substar $S'$ obtained in this way  is called the {\it core of $S$},
its set of edges is denoted by $\kr(S)$, and by $\krr(S)$ we mean the set of vertices 
forming its body.
	
The first two parts of the following fact list standard properties of the process 
by means of which the core is constructed, while its third part states that
cores have a property reminiscent of condition~\ref{it:VSS} in Definition~\ref{dfn:family}.   

\begin{fact}\label{fact:core} 
	Let $r\ge 1$ and $t\ge 2$. 
	\begin{enumerate}[label=\alabel]
		\item\label{it:f1} There are integers $n_0$ and $c_0$ such that if a hypergraph
			$H\in \irt$
			has $n\ge n_0$ vertices and~$m>\binom{\lfloor n/t\rfloor}{rt}$ 
			edges, then there is a heavy star $S\subseteq H$ with
						\[
				|S|\ge \frac{m\hn(n, m)}{n}-3r^2t^3n^{rt-1}
				\quad \text{ and } \quad
				|\krr(S)|\ge \hn(n, m)-c_0>\frac{2n}{5t}\,,
			\]
						where $\hn(n, m)$ is the number introduced in~\eqref{eq:hn}.
		\item\label{it:f2} For every positive integer $a$ there exists an integer $b$ such 
			that every $\bigl(rt(t-1)\bigr)$-star~$S$ with $N$ vertices in its body 
			and $|S|\ge\binom{N-a}{rt}$ satisfies $|\krr(S)|\ge N-b$.  
		\item\label{it:f3} If $H\in\irt$ and $S\subseteq H$ is a heavy 
			$\bigl(rt(t-1)\bigr)$-star, then every edge of $H$ intersecting the 
			body of $S$ needs to contain the centre of $S$. 
		\end{enumerate}
\end{fact}

\begin{proof}
	For the proof of part~\ref{it:f1} we take $n_0$ to be at least as large 
	as the number provided by Theorem~\ref{thm:mainr3}. We then know that 
	for any $H\in\irt$ as above there exists a star $\hat S\subseteq H$ 
	with $|\hat S|\ge \frac{m\hn(n, m)}{n}-n^{rt-1}$. 
		Throughout the process yielding $S=\kr(\hat S)$ 
	we remove at most $2r^2t^3\sum_{i=1}^{n}i^{rt-2}< 2r^2t^3n^{rt-1}$
	edges and thus $S$ has at least the size we claimed. 
	To obtain the desired lower bound on $|\krr(S)|$ we observe that the definition 
	of $\hn=\hn(n, m)$ implies $\frac{(\hn-1)m}n>\binom{\hn-1}{rt}$, whence
		\[
		\binom{|\krr(S)|}{rt}\ge |S| > \binom{\hn-1}{rt}-3r^2t^3n^{rt-1}
		>\binom{\hn-c_0}{rt}  
	\]
		holds for sufficiently large $c_0$ and $n_0$.
	
	For the verification of part~\ref{it:f2} we may take two large constants 
	$b$, $b'$ with $b\gg b'\gg a$. There is nothing to prove in case $N\le b$,
	so let us assume $N> b$ from now on. 
					As above we have $|S\sm \kr(S)|\le 2r^2t^3N^{rt-1}$ and, hence, 
		\[
				|\kr(S)|\ge \binom{N-a}{rt}-2r^2t^3N^{rt-1}
		\ge\binom{N-b'}{rt}\,,
	\]
		which is only possible if $|\krr(S)|\ge N-b'$.
	
	Finally, let $H$ and $S$ be as in~\ref{it:f3}, write $B$ for the set of vertices 
	forming the body of~$S$, set $N=|B|$, and consider any $e\in H$ intersecting $B$ 
	in some vertex $v$. The minimum degree condition satisfied by $S$ yields 
	$|S|\ge 2rt^2N^{rt-1}$. As at most $rt^2N^{rt-1}$ edges of $S$ can intersect $e\cap B$,
	there is an edge $f\in S$ disjoint to $e\cap B$ and, consequently, we have $|e\cap s|\in I$,
	where $s$ denotes the centre of $S$.
	Similarly, for every $w\in e\cap B$ distinct from $v$ there are at most $N^{rt-2}$
	edges of $S$ containing both $v$ and $w$, and thus there is an edge $f'\in S$
	with $(e\cap B)\cap f'=\{v\}$, which proves $|e\cap s|+1\in I$. But the only possibility 
	for the consecutive integers $|e\cap s|$ and $|e\cap s|+1$ to belong to $I$ is that
	$s\subseteq e$, as desired. 
\end{proof}

The following result describes the structure of dense $H\in \EX(\irtt; n, m)$
quite precisely. 

\begin{theorem}\label{prop:25}
	For all integers $t\ge 2$ and $r\ge 1$, 
	there exist an integer~$n_*=n_*(r,t)$ and  constants $c_i=c_i(r,t)$, $i=1,2,3$,
	such that from every $rt^2$-graph $H=(V,E)$ in $\EX(\irtt;n,m)$ with $n\ge n_*$ vertices 
	and $m>\binom{\lfloor n/t\rfloor}{\el}$ edges one can remove at most $c_1$ edges 
	to get a graph which consists of $\ell\le 7t^2$ many $\bigl(rt(t-1)\bigr)$-stars
	 $S^1,S^2,\dots, S^\ell$ and, perhaps, some number of isolated vertices.
	   
	Moreover, we can also assume that 
	\begin{enumerate}[label=\rmlabel]
		\item\label{it:11}$|\krr(S^i)|\ge n/(7t^2)$, for $i=1,2,\dots, \ell-1$;
		\item\label{it:22} $\big|V\setminus \bigcup_{i=1}^\ell \krr(S^i) \big|\le c_2$;
		\item\label{it:33}the centres of the stars $S^1, \dots, S^{\ell}$ are pairwise disjoint. 
			\end{enumerate}

	In particular, we can delete from $H$ at most $c_3n^{(r-1)t}$ edges to get a 
	union of at most $\ell$ vertex disjoint $\bigl(rt(t-1)\bigr)$-stars. 
\end{theorem} 

\begin{proof}
Let us assume that $n$ is sufficiently large, $m>\binom{\lfloor n/t\rfloor}{\el}$, and 
that $H\in \EX(\irtt;n,m)$. The idea for constructing the first $\ell-1$ of the desired stars 
is to apply Fact~\ref{fact:core}\ref{it:f1} iteratively, pulling these stars out of $H$ one 
by one. This process comes to an end when we cannot guarantee anymore to find a star with 
a sufficiently large core in the remaining part of~$H$. Then we argue that the remaining part 
of $H$ which lies outside the stars cannot be large. Otherwise we could delete one edge from 
each large star and all edges of $H$ which do not belong to them and create a new star, 
disjoint from the one already found.  In this way out of $H$ we could construct a new graph 
$H'\in\irt $ with the same number of vertices and edges as $H$ but which has smaller maximum 
degree  contradicting the fact that $H\in \EX(\irtt; n, m)$. A similar argument (finding 
$H'\in \irt$ with $\Delta(H')<\Delta(H)$) shows that the centres of the large stars must 
be disjoint. 

Let us make the above argument precise. 
Set $H^1=H$. Due to Fact~\ref{fact:core}\ref{it:f1} 
there exists a maximal star $S^1\subseteq H^1$ with $|\krr(S^1)|\ge 2n/(5t)$ and 
$|\kr(S^1)|\ge m\hn(n, m)/n-3r^2t^3n^{rt-1}$.
By $H^2$ we denote the hypergraph arising from $H^1$ by the deletion of all 
vertices in $\krr(S^1)$ and all edges in~$S^1$. Note that, by Fact\ref{fact:core}\ref{it:f3},
all edges which intersected removed vertices belonged to removed edges, i.e.\ we deleted from $H_1$
only edges from $S^1$.
If for some integer $i\ge 2$ we have 
just chosen a star $S^{i-1}\subseteq H^{i-1}$ and constructed a hypergraph~$H^i\subseteq H^{i-1}$
 with $n_i$ vertices and $m_i$ edges,
we check whether the conditions
\begin{enumerate}[label=\alabel]
	\item\label{it:Sa} $n_i\ge \hn(n,m)$,
	\item\label{it:Sb} $m_{i}> \binom{\lfloor n_{i}/t\rfloor}\el$
\end{enumerate}
are satisfied. If at least one of them fails we set $\ell=i$ and terminate the procedure,
the last constructed objects being $S^{\ell-1}$ and $H^\ell$.
On the other hand, if both conditions hold, the assumptions of Fact~\ref{fact:core}\ref{it:f1}
are satisfied by $H^i$. Thus we find a maximal star $S^i\subseteq H^i$ 
with $|\krr(S^i)|\ge 2n_i/(5t)$ and $|\kr(S^i)|\ge m_i\hn(n_i, m_i)/n_i-3r^2t^3n_i^{rt-1}$.  
Moreover, we let~$H^{i+1}$ denote the hypergraph with vertex set $V(H^i)\sm \krr(S^i)$ 
and edge set $H^i\sm S^i$.

Notice that for each $i\in [\ell-1]$ property~\ref{it:Sa} of the above process entails  
\[
	|\krr(S^i)|\ge \frac{2n_i}{5t}\ge \frac{2\hn(n, m)}{5t}\ge \frac{4n}{25t^2}
	>\frac{n}{7t^2}\,,
\]
meaning that condition~\ref{it:11} of the theorem holds. Besides, since 
$\krr(S^1), \ldots, \krr(S^{\ell-1})$ are mutually disjoint, it also 
follows that $\ell\le 7t^2$. Denote the centre of $S^i$ by $s_i$ for $i\in [\ell-1]$.
Note also that,  by Fact~\ref{fact:core}\ref{it:f3}, in the process of deleting vertices 
of $\krr(S^i)$, we have  destroyed no edges other than that of $S^i$. 
Thus, not only  $\bigcup S_i\cup H^\ell \subset H$ but, in fact, $\bigcup S_i\cup H^\ell = H$.

Now we study the structure of the last hypergraph $H^\ell$. 
The following result is crucial for our argument. It shows, in particular, that 
our statement about $c_1$ holds.

\begin{claim}\label{cl:pr2}
	There are constants $c_1=c_1(r,t)$ and $c_4=c_4(r,t)$ such that after removing 
	at most $c_1$ edges from $H^\ell$ this hypergraph becomes the union of 
	an $\bigl(rt(t-1)\bigr)$-star with at least 
	$\binom{n_{\ell}-c_4}{rt}$  edges and, perhaps, some number of isolated vertices.
\end{claim}

\begin{proof}
	Suppose that the assertion does not hold. Our aim is to get a contradiction with 
	the assumption that $H\in \EX(\irtt;n,m)$ by constructing a graph $H'\in \irt$ 
	having the same number of vertices and edges as $H$ but a smaller maximum degree.
	Choose some absolute constants $c$, $c'$, $c''$, $c_1$, and $c_4$ depending only 
	on $r$ and $t$, sufficiently large so that all arguments below will work, 
	and obeying the hierarchy
		\[
		c_1\gg c_4\gg c''\gg c'\gg c\,.
	\]
	
	{\it \hskip1em Case 1: $n_\ell < \hn(n, m)$.}

	\smallskip

	Notice that we may assume $n_\ell\ge c_4$, since otherwise an appropriate choice
	of $c_1$ would show that the claim holds with an empty star. Moreover, 
	if $m_\ell>\binom{n_\ell-c''}{rt}$ the desired conclusion can be drawn from 
	Proposition~\ref{prop:silly} (and  Theorem~\ref{thm:dt}). 
	So we may suppose~${m_\ell\le \binom{n_\ell-c''}{rt}}$ from now on.

    The hypergraph $H'$ will have three kinds of edges. First, there will
    be stars $\hat S^1, \dots, \hat S^{\ell-1}$, where each $\hat S^i$  
    is obtained from $\kr(S^i)$ by the omission of a single edge. 
    
    Second, there will be edges serving as ``substitutes'' for the edges in $F^i=S^i\sm\kr(S^i)$ 
    for $i\in [\ell-1]$. The reason for this substitution is that it ``cleans up some space''
    so that in the end $H'\in\irt$ will be true. Let us recall that $|F^i|\le O(n^{\el-1})$
    follows from the construction of cores. Hence, there are disjoint 
    subsets $U^1, \ldots, U^{\ell-1}$ of the vertex set of~$H^\ell$ with 
    $|U^i|=\big\lceil |F^i|/\binom{|\krr(S^i)|}{rt-1}\big\rceil\le c$ for~$i\in[\ell-1]$.  
    Now instead of $F^i$ we put the same number of edges of the type $\{v\}\cup s_i\cup f$ 
    into $H'$, where $v\in U^i$, and 
    $f$ is a subset of the body $\krr(S^i)$ with $\el-1$ elements.   
     
    Third, we include a star with $m_\ell+(\ell-1)$ edges into $H'$ that 
    uses only vertices of $H^\ell$ that are not occupied by the sets $s_i$ and $U^i$ 
    for $i\in [\ell-1]$. There is enough space for such a star, as at most $\ell(rt^2+c)\le c'$
    vertices are occupied, ${m_\ell\le \binom{n_\ell-c''}{rt}}$, and $n_\ell$
    is sufficiently large.
    
    It remains to check that we have indeed $\Delta(H')<\Delta(H)$. The only 
    vertices of $H'$ that might be problematic are in the centre of the new star 
    that has just been created.
    
    However, working carefully with the estimates provided by Theorem~\ref{thm:mainr3}
    and exploiting that we are in the first case, one checks easily that 
        \begin{align*}
    	\Delta(H)&\ge |S^1|\ge\frac{m\hn(n, m)}{n}-3r^2t^3n^{rt-1}\ge \binom{\hn(n, m)-c'}{rt} \\
		&>\binom{n_\ell-c''}{rt}+(\ell-1)\ge m_\ell+(\ell-1)\,.
    \end{align*}
        Thus $H'$ contradicts indeed our assumption that $H\in \EX(\irtt;n,m)$.

	\medskip\goodbreak

	{\it \hskip1em Case 2 : $n_\ell \ge \hn(n, m)$.}

	\smallskip

	This means that the iterative procedure that led us to the stars $S^1, \ldots, S^{\ell-1}$
	stopped owing to the failure of condition~\ref{it:Sb}, i.e., that 
		\begin{equation}\label{eq:ml}
		m_{\ell}\le \binom{\lfloor n_{\ell}/t\rfloor}\el\,.
	\end{equation}
		One can deal with this case 
	in a very similar way as with the previous one 
	but instead of replacing the edges of $H^\ell$ by one large star we replace them 
	by $t$ smaller and mutually disjoint stars with roughly $n_\ell/t$ 
	vertices and $m_\ell/t$ edges.
	
	The inequality~$\frac{m_\ell}{t}\le \binom{n_{\ell}/t-c'}\el$, 
	which is a direct consequence of~\eqref{eq:ml}, shows that there is indeed 
	enough space for such stars.
	 
	Finally, it remains to check that the hypergraph $H'$ generated as above 
	really has a smaller maximum degree than $H$. 
	
	If $n_\ell\le \frac{tn}{t+1}$ this follows from 
		\begin{align*}
		\left\lceil\frac{m_\ell+(\ell-1)}{t}\right\rceil & 
		\le \frac 1t\binom{n/(t+1)}{rt}+O(1)
		\le \frac 1t\left(\frac t{t+1}\right)^{rt}\binom{n/t}{rt}+O(n^{rt-1}) \\
		&\le \frac {4m}{9t} +O(n^{rt-1})
		< \frac{m\hn(n, m)}{n}-3r^2t^3n^{rt-1}\le |S^1|\,.
	\end{align*}
		The case $\frac{tn}{t+1}<n_\ell$, however, is impossible, because due to 
	$n_\ell\le n-\krr(S^1)<\bigl(1-\frac 2{5t}\bigr)n$ it would entail
		\begin{align*}
		m &= \sum_{i=1}^{\ell-1}\bigl(|\kr(S^i)|+|F^i|\bigr)+ m_\ell
		 \le \sum_{i=1}^{\ell-1}\binom{|\krr(S^i)|}{rt}+\binom{\lfloor n_{\ell}/t\rfloor}{rt}
		 +O(n^{rt-1}) \\
		 & \le \binom{n-n_\ell}{rt}+\binom{n_{\ell}/t}{rt}+O(n^{rt-1})
		 <\binom{\lfloor n/t\rfloor}{rt}\,,
	\end{align*}
	where for the last estimate we used that $(1-x)^{rt}+(x/t)^{rt}<(1/t)^{rt}$
	holds for all real $x\in\bigl[\frac{t}{t+1}, 1-\frac 2{5t}\bigr]$.  
	But the above estimate contradicts our initial hypothesis about $m$.
\end{proof}

	By Claim~\ref{cl:pr2} there exists a constant $c_4=c_4(r,t)$ such that $H^\ell$ contains an $(rt(t-1))$-star, call it $S^\ell$, with at least $\binom{n_\ell-c_4}{rt}$ edges as a subgraph. 
	We can therefore apply Fact~\ref{fact:core}\ref{it:f2} and argue that 
	there exists a  constant $c_2$ such that $|\krr(S_\ell)|\ge n_\ell-c_2$. 
Hence $|V\setminus \bigcup^\ell_{i=1} \krr(S^i)|\le c_2$ and~\ref{it:22} follows.
Let $s_\ell$ be the centre 
of $S^\ell$.

In order to verify~\ref{it:33} let us observe first that 
each star $S^i$ consists of at most 
\[
	\binom{|\krr(S_i)|}{rt} + O(n^{rt-1}) 
		\]
edges. It may be helpful to recall each of the sets $\krr(S^1), \ldots, \krr(S^{\ell-1})$
has size $\Omega(n)$. We do not know the same about the last star, but at least 
we may suppose that $\krr(S^\ell)$ is sufficiently large for otherwise we may ignore this 
star and proceed. Now let us assume that there are two stars, $S^i$ and $S^j$, 
$1\le i<j\le \ell$, which do not have disjoint centres. Then we construct a new 
hypergraph $H'\in \irt $ out of $H\in \EX(\irtt;n,m)$ in the following way. We delete 
all the edges of the stars $S^i$ and~$S^j$, say of $m'_i$ and $m'_j$ edges respectively, 
and on the vertex set $\krr(S^i)\cup \krr(S^j)$ we create an $\bigl(rt(t-1)\bigr)$-star $S^{ij}$
which has $m'_i+m'_j-1<\Delta(H)$ edges 
and which uses as few vertices as possible. Due to $|\krr(S^i)|\ge\Omega(n)$
and $|\krr(S^j)|\ge \Omega(1)$ we have 
\[
	\binom{|\krr(S^i)|}{rt}+\binom{|\krr(S^j)|}{rt}+O(n^{rt-1})
	<\binom{|\krr(S^i)|+|\krr(S^j)|-8rt^4}{rt}\,,
\]
and hence $S^{ij}$ uses fewer
than $|\krr(S^i)|+|\krr(S^j)|-7rt^4$ vertices. Now we remove one edge from each of the 
existing stars
$S^t$, where $t=1,2,\dots,\ell$, and $t\neq i,j$, and add $\ell-1\le 7t^2$
disjoint edges to the hypergraph. Such a hypergraph $H'\in \irt$ has a smaller maximum degree 
than $H$, which contradicts the fact that $H\in \EX(\irtt;n,m)$. 
Thus, the centres of the stars $S^i$, $i=1,2,\dots, \ell$, are pairwise 
disjoint, as claimed in clause~\ref{it:33} of the theorem. 
For later use we record that using Fact~\ref{fact:core}\ref{it:f3} one can show 
that $s_i\cap\krr(S^j)=\varnothing$ holds whenever $i, j\in [\ell]$.

Finally, we need to argue that we can delete from $H\in \EX(\irtt;n,m)$ at most
$c_3 n^{(r-1)t}$ additional edges to make $S^1,\dots, S^\ell$ vertex disjoint. 
We contend that it suffices for this purpose to delete all edges intersecting
the set $W=\bigcup_{i=1}^\ell \krr(S^i)$ in at most $(r-1)t$ vertices. 
Notice that due to~\ref{it:22} the resulting hypergraph $\widetilde H$ 
differs in at most $c_2^{rt^2}n^{(r-1)t}$ edges from~$H$. Owing to 
of Fact~\ref{fact:core}\ref{it:f3} and~\ref{it:33} there is for every edge $e\in \widetilde H$
a unique~$i\in [\ell]$ with $e\in S^i$; for this $i$ we have 
$|e\cap \krr(S^i)|\ge (r-1)t+1$ and, hence, $e$ has at most $t-1$ vertices outside
$s_i\cup\krr(S^i)$. These vertices cannot belong to the centre $s_j$ of another star,
for then~$e$ would have a forbidden intersection with every edge in $\kr(S^j)$. 
Now assume that one of the vertices in $e\sm\bigl(s_i\cup \krr(S^i)\bigr)$, say $v$,
would belong in $\widetilde H$ to another star as well. This means that there 
are an index $j\ne i$ and an edge $f\in S^j\cap \widetilde H$ with $v\in e\cap f$.
Due to $f\in \widetilde H$ at most $t-1$ vertices of $f$ are outside $s_j\cup \krr(S^j)$
and thus we have $1\le |e\cap f|\le t-1$, which is absurd. Thus $\widetilde H$ 
is indeed a union of $\ell$ vertex disjoint stars.
 \end{proof}
 
Let us comment briefly on the structure of $H\in \EX(\irtt;n,m)$ described in 
Theorem~\ref{prop:25}.
Once we know Theorem~\ref{prop:25} the estimate for the number of stars $\ell$ 
can be easily improved to the optimal $\ell\le \lceil t^{rt/(rt-1)}\rceil$ (see \cite{LP}, 
where a similar argument is used for $r=1$, $t=2$). 
However, we cannot significantly decrease the number of edges needed to make the stars vertex 
disjoint.  To see this, let us consider the $rt^2$-graph $\tilde H\in \EX(\irtt;n,m)$ 
with vertex set $V=U_1\dcup U_2\dcup C_1 \dcup C_2\dcup T$, where $|U_1|=|U_2|=u$, 
$|C_1|=|C_2|=rt(t-1)$, $|T|=t$, whose set of edges consists of:
\begin{itemize}
\item $\binom{u}{rt}$ subsets which are unions of $C_1$ and some $rt$-element subset of $U_1$,
\item $\binom{u}{rt}$ subsets which are unions of $C_2$ and some $rt$-element subset of $U_2$,
\item $\binom{u}{(r-1)t}$ subsets which are unions $T$, $C_1$ and some $(r-1)t$-element subset 
of $U_1$,
\item $\binom{u}{(r-1)t}$ subsets which are unions $T$, $C_2$ and some $(r-1)t$-element subset 
of $U_2$,
\item and a thick clique on $T\cup C_1\cup C_2$ whose teams are $T$ and partitions 
of $C_1$, $C_2$.
\end{itemize}

It is easy to see that $\tilde H\in \EX(\irtt;n,m)$ with $n=2u+2rt(t-1)+t$ and the 
appropriate~$m$.
On the other hand, up to a thick clique of bounded size, $\tilde H$ consists of two stars 
with centres $C_1$ and $C_2$ and 
to make them vertex disjoint one must delete at least $\Omega(n^{(r-1)t})$ edges.

Finally, let us notice that from Theorem~\ref{prop:25} it follows that almost all the 
edges of dense extremal graphs from $\EX(\irtt;n,m)$ are contained in at most $\ell$ stars 
among which $\ell-1$ are roughly equal and only one can be a bit smaller than the others. 
Having this in mind one can easily compute  the scaled extremal function 
\[
	 f_{rt}(x)=\lim_{n\to\infty}
	 \frac{\ff(\irtt; n,x\binom{n-rt(t-1)}{\el})}{\binom{n-rt(t-1)}{\el}}\,.
\]
From Theorems~\ref{thm:mainr3} and~\ref{prop:25} we know that the function is well defined
for $x\in [0,1]\setminus \{t^{-tr}\}$. Furthermore, besides the point $x=t^{-tr}$ where it
jumps from $0$ to some value which is at least $t^{-t^2r^2/(rt-1)}$, it is continuous everywhere. It is also 
smooth everywhere except the points $x=j^{1-tr}$ for 
$j=2,3,\dots,  \lceil t^{tr/(rt-1)}\rceil-1$ (see \cite{LP} where details are worked out for 
the case  $r=1$, $t=2$).

\section{Tools}\label{sec:tools}

The purpose of this section is to gather three statements that will turn out to be 
useful in the proof of the Structure Theorem. While the first two of them are fairly 
well known, the third one (see Lemma~\ref{lem:dl} below) could very well be new.

\subsection{Divisible set systems} Given a natural number $t\ge 2$ we shall say that 
a set system~$(V, \cE)$ is {\it $t$-divisible} if for any two distinct edges $e, e'\in \cE$
the size $|e\cap e'|$ of their intersection is a multiple of $t$. The problem to study
upper bounds on the size of such set systems with additional assumptions on the 
behaviour of the sizes of the edges modulo $t$ was first studied, in the particular case $t=2$,
by Berlekamp~\cite{Be69}, who realised that ideas pertaining to linear algebra can be applied 
in such contexts. At a later occasion we will need a variant of one of his results that was first 
observed, in a more general form, by Babai and Frankl (see~\cite{BF80}*{Theorem~1}).

\begin{lemma}\label{lem:linalgt}
	Let $(V, \cE)$ be a $t$-divisible set system for some natural number $t\ge 2$. 
	If $|e|\equiv 1 \pmod{t}$ holds for all $e\in\cE$, then $|\cE|\le |V|$.
\end{lemma}

\begin{proof}
	Let $p$ denote a prime factor of $t$. We identify the members of $\cE$ with vectors 
	from the $|V|$-dimensional vector space $\mathbbm{F}_p^V$ via characteristic functions
	and contend that the stronger conclusion that $\cE$ is linearly independent holds.
	To see this, one looks at a hypothetical linear dependency 
	$\alpha_1e_1+\dots+\alpha_ne_n=0$ with distinct $e_i\in \cE$ and certain numbers
	$\alpha_i\in \mathbbm{F}_p\setminus \{0\}$, where $n\ge 1$.
	Taking the standard scalar product with~$e_1$ we obtain 
	\[
		0=\langle e_1,0\rangle=\langle e_1,\alpha_1e_1+\dots+\alpha_ne_n\rangle
		=\sum_{i=1}^{n}\alpha_i\langle e_1,e_i\rangle = \alpha_1,
	\]
	which is absurd.
\end{proof}

\subsection{Delta systems}
A set system $\F$ is called a {\em sunflower} (or a {\em $\Delta$-system}) if there exists 
a (possibly empty) set  $S$ of vertices such that the intersection of any two distinct 
edges of~$\F$ is equal to $S$. This constant intersection $S$ is called the {\em kernel} 
of the sunflower. 

In~1960 Erd\H os and Rado~\cite{ER} proved their ``sunflower lemma'' saying 
that any sufficiently large collection of finite sets of bounded size contains 
big sunflowers. 

\begin{theorem}\label{thm:sunflower}
	For all positive integers $a$ and $b$, any collection of more than $b!a^{b+1}$ sets of 
	cardinality at most $b$ contains a $\Delta$-system with more than $a$ elements. 
\end{theorem}

It should perhaps be pointed out that $b!a^{b+1}$ is not the least number $f(a, b)$ for which
this statement is true. In fact, Erd\H{o}s and Rado themselves stated a marginally better 
but less clean upper bound on this number in~\cite{ER}*{Theorem III}, but despite the 
considerable attention that the problem to improve our understanding of the growth behaviour 
of this function has received (see e.g.,~\cite{Ko97}) the progress on this problem has 
been rather slow. For the purposes of the present article, however, even knowing the exact 
value of $f(a, b)$ would be quite immaterial.
\subsection{Divisible pairs of set systems}

The next result makes use of the following concept. 

\begin{definition}
	Let $\F$ and $\G$ denote two set systems with the same vertex set $V$ and let~$q$ be 
	a positive integer. We say that the pair $(\F, \G)$ is {\it $q$-divisible} if for all 
	$f\in \F$ and $g\in \G$ the size $|f\cap g|$ of their intersection is divisible by $q$. 
\end{definition}

Let us emphasise that the edges $f$ and $g$ occurring in this definition are not required 
to be distinct. In other words, if $e\in \F\cap\G$, then $|e|$ needs to be divisible by $q$.
 
The result that follows will often help us to analyse the structure of divisible pairs.  

\begin{lemma}[Decomposition lemma]\label{lem:dl}
	Suppose that $k$, $q$ are positive integers and that $\F, \G$ are two set systems 
	with the same vertex set $V$ such that 
	\begin{enumerate}
		\item[$\bullet$] all members of $\F\cup \G$ have size at most $k$,  
		\item[$\bullet$] and the pair $(\F,\F\cup \G)$ is $q$-divisible.
	\end{enumerate}
	Then there is a set system $\cH$ on $V$ with the following properties:
	\begin{enumerate}[label=\rmlabel]
		\item\label{it:dc1} $\Delta(\cH)\le k^{2k}$;
		\item\label{it:dc2} $\cH$ is an antichain (that is, $x\not\subseteq y$ holds 
			for all distinct $x, y\in\cH$);
		\item\label{it:dc3} Every edge of $\F$ is a disjoint union of edges from $\cH$;
		\item\label{it:dc4} The pair $(\cH,\cH\cup \G)$ is $q$-divisible.
	\end{enumerate}
\end{lemma}

\begin{proof}
	Without loss of generality we may assume that given $k$, $q$, and $\G$, the set system $\F$ 
	is maximal with respect to inclusion,  i.e., that for every set system $\F^*\supsetneqq \cF$ 
	with $|f|\le k$ for all $f\in \F^*$ the pair $(\F^*,\F^*\cup \G)$ fails to be $q$-divisible.
	
	Now we define $\cH$ to be the collection of those members of $\cF\sm\{\vn\}$ that are 
	minimal with respect to inclusion, i.e., we set 
		\[
		\cH=\{h\in \F\sm\{\vn\}\colon \text{ if } f\in \F \text{ and }f\neq \vn, h, 
		\text{ then } f\nsubseteq h\}\,.
	\]
	
	This choice of $\cH$ makes part~\ref{it:dc2} obvious and~\ref{it:dc4} follows directly
	from $\cH\subseteq \F$. 
	
	Assuming that~\ref{it:dc3} would be false let $f\in\cF$ be chosen with $|f|$ minimum 
	such that~$f$ is not expressible as a disjoint union of appropriate edges from $\cH$.
	Since the empty set is equal to the empty union, we have $f\ne\vn$. Moreover, $f$ 
	cannot belong to $\cH$ and, consequently, there exists some $h\in \cH$ with $h\subseteq f$
	and $0< |h|< |f|$. Notice that for every~$g\in\F\cup\G$ the number 
	$|(f\sm h)\cap g|=|f\cap g|-|h\cap g|$ is divisible by~$q$. Besides $|f\sm h|=|f|-|h|$
	is divisible by $q$ as well. Owing to the maximality
	of $\F$ it follows that $(f\sm h)\in\F$. But in view of our minimal choice of~$f$ 
	this means that $f\sm h$ is a disjoint union of edges from $\cH$ and, hence, so is~$f$.
	Thus $\cH$ must satisfy~\ref{it:dc3}. 
	 
    Now it remains to show that $\cH$ has bounded maximum degree. Assume for the sake  	
    of contradiction, that there exists a vertex $v\in V$ contained in more than $k^{2k}$ 
    edges of~$\cH$ and look at the set system
     \[
    	\F_v=\{h\in\cH\colon  v\in h\}\,.
	\]
	
	In view of $|\F_v|> k^{2k}\ge k!k^{k+1}$ Theorem~\ref{thm:sunflower} reveals that 
	$\F_v$ contains a $\Delta$-system $\F^*$ with more than $k$ elements. Denote the kernel 
	of $\F^*$ by $e$ and observe that, since $|\F^*|>k$, for all edges $g\in \F\cup \G$, 
	the size of the intersection $|e \cap g|$ is divisible by $q$. Moreover, by~$|\F^*|\ge 2$
	again, one can express $e$ as the intersection of two members of $\F^*$, 
	whence $|e|$ is divisible by $q$ as well. 
	
	Together with the maximality of $\F$ these facts imply $e\in \F$. 
	Using $|\F^*|\ge 2$ again we get some 
	$h\in \F^*\subseteq \cH$ properly containing $e$ and by our definition of $\cH$ this 
	is only possible if $e=\vn$. But, on the other hand, we certainly have $v\in e$. 
	This contradiction concludes the proof of~\ref{it:dc1} and, hence, the proof of the 
	decomposition lemma.
\end{proof}

\section{Proof of the Structure Theorem} \label{sec:4}

This entire section is dedicated to the proof of Theorem~\ref{thm:dt}.
Let integers $r\ge 1$ and~$t\ge 2$ as well as an $rt^2$-uniform hypergraph
$H=(V,E)$ with $|V|=n$ be given such that 
the size of the intersection of any two edges of $H$ belongs to 
the set 
\[
	I=\{s\colon t\mid s \textrm{ or } s\ge rt(t-1)\}\,,
\]
which means $H\in \irt$. 
We shall show that $H\in \F\bigl(t,rt,(rt^2)^{r^3t^6}\bigr)$.

Let us start by colouring all those subsets $f\subseteq V$ with $|f|\le rt(t-1)+1$ 
{\it red} that are kernels of sunflowers consisting of at least $rt^2$ 
edges of $H$. Recall that the latter condition means that there are to exist 
$rt^2$ disjoint sets $f_1, \dots, f_{rt^2}\subseteq V$ of 
size $rt^2-|f|$ such that $f\cup f_i\in H$ holds for every $i\in [rt^2]$. 
We denote the set system on $V$ whose edges are the red sets by $H_{\text{red}}$.
By $H^*_{\red}$ we mean the $\bigl(rt(t-1)+1\bigr)$-uniform hypergraph on $V$ 
whose edges are the red $\bigl(rt(t-1)+1\bigr)$-sets and finally we put
$\widehat{H}_{\red}=H_{\red}\setminus H^*_{\red}$.

Observe that  
\begin{equation}\label{eq:rh}
	\text{for any (not necessarily distinct) }
	f, f'\in H\cup H_{\red}
	\text{ we have }
	|f\cap f'|\in I\,.
\end{equation}
This is because we can first extend $f$ to an edge $e$ of $H$ with $e\cap f'=f\cap f'$
and proceeding similarly with $f'$ we get an edge $e'\in H$ with $e\cap e'=f\cap f'$,
so that $|f\cap f'|=|e\cap e'|\in I$ follows from the assumption that $H$ be $I$-intersecting.

As a consequence of this observation we learn that for any distinct $f, f'\in H^*_{\red}$
the number~$|f\cap f'|$ is divisible by $t$ and in view of Lemma~\ref{lem:linalgt} it follows
that  
\begin{equation}\label{eq:h7}
	|H^*_{\red}|\le n\,.
\end{equation}

Moreover,~\eqref{eq:rh} reveals that the pair 
$\bigl(\widehat{H}_{\red}, \widehat{H}_{\red}\cup H^*_\red\cup H\bigr)$
is $t$-divisible, which allows us to apply the decomposition lemma 
(Lemma \ref{lem:dl}) to $rt^2$, $t$, $\widehat{H}_{\red}$, and $H^*_\red\cup H$
here in place of~$k$,~$q$, $\F$, and $\G$ there. We thus infer the existence 
of a set system $G$ on $V$ with the following properties:
\begin{enumerate}[label=\alabel]
	\item\label{dl:dc1} $\Delta(G)\le (rt^2)^{2rt^2}$;
	\item\label{dl:dc2} $G$ is an antichain;
	\item\label{dl:dc3} Every edge of $\widehat{H}_{\red}$ is a disjoint union of edges from $G$;
	\item\label{dl:dc4} The pair $(G, G\cup H^*_\red \cup H)$ is $t$-divisible.
\end{enumerate}

We imagine that the edges of $G$ have been coloured {\it green}. 
The green sets of cardinality~$t$ will be referred to as {\it teams}. 
Notice that due to condition~\ref{dl:dc4} the teams are inseparable in~$H$ and,
moreover, by~\ref{dl:dc2} and~\ref{dl:dc4} each team is disjoint to any other green set.

Now we are ready to decompose $V$ and $H$ in the envisioned way. 
We start by defining~$V_T$ to be the union 
of all teams and setting $H_T=H[V_T]$, which guarantees part~\ref{it:VR} of 
Definition~\ref{dfn:family}.

Preparing the definition of $H_S$ we colour a set consisting of $rt(t-1)$ vertices {\it purple} 
if it is the kernel of a $\Delta$-system in $H^*_\red$ of size at least $rt^2$.
Imitating the proof of~\eqref{eq:rh} one checks easily that 
\begin{equation}\label{eq:central}
	\text{ if } Y, Y' \text{ are purple and } f\in H^*_\red\cup H, 
	\text{ then } |Y\cap Y'|, |Y\cap f|\in I\,.
\end{equation}

Now we define $H_S$ to be the collection of all edges $h\in H$ which contain a 
purple set~$Y_h\subseteq h$ such that
$Y_h\cup\{v\}\in H_\red^*$ holds for each~$v\in h\sm Y_h$.
Moreover, we set 
\[
	V_S=\bigcup_{h\in H_S}(h\setminus Y_h)\,,
\]
and contend that 
\begin{equation}\label{eq:STdisj}
	V_S\cap V_T=\varnothing\,.
\end{equation}

Otherwise, there would exist a vertex $v\in V_S\cap V_T$, meaning that there are
an edge $h\in H_S$ with $v\in (h\sm Y_h)$ and a team $g$ with $v\in g$. Applying~\ref{dl:dc4}
to $g\in G$ and $(Y_h\cup\{v\})\in H_\red^*$ we obtain $g\subseteq Y_h\cup\{v\}$. 
By $|h\sm Y_h|=rt\ge 2$ there is a vertex $w\in h\sm Y_h$ distinct from $v$. 
As the set $Y_h\cup\{w\}$ belongs to $H_\red^*$ and intersects $g$ in $t-1$ vertices, 
we get a contradiction to~\ref{dl:dc4}, which proves~\eqref{eq:STdisj}.

Now provided we can show
\begin{equation}\label{eq:VS}
	Y_h\cap V_S=\varnothing 
	\quad \text{ for each } 
	h\in H_S
\end{equation}
it will be clear that $H_S$ is a union of stars with centres $Y_h\subseteq (V\sm V_S)$
and their bodies in~$V_S$, as required by Definition~\ref{dfn:family}\ref{it:VS}.

For the proof of~\eqref{eq:VS} we assume indirectly that for some $h\in H_S$ 
there is a vertex $v\in Y_h\cap V_S$. This means that there exists an edge $h'\in H_S$ 
with $v\in h'\sm Y_{h'}$. But now $|Y_h\cap Y_{h'}|$ and $|Y_h\cap (Y_{h'}\cup\{v\})|$
are two consecutive integers belonging to $I$ by~\eqref{eq:central} and both are at most 
$|Y_h|=rt(t-1)$, contrary to $t\ge 2$. Thereby~\eqref{eq:VS} is proved.    

Next we observe that if for some $f\in H$ and $h\in H_S$ there is a vertex 
$v\in f\cap h\cap V_S$, then the consecutive integers $|f\cap Y_h|$ and 
$|f\cap (Y_h\cup \{v\})|$, again by \eqref{eq:central}, are both in $I$. 

Consequently, 
\begin{equation}\label{eq:Asia's-trick}
	\text{if $f\in H$ and $h\in H_S$ satisfy 
	$f\cap h\cap V_S\neq \emptyset$, then $Y_h\subseteq f$.}
\end{equation}

Hence, all stars in $H_S$ must be semi-disjoint and we may associate with each 
vertex~$v\in V_S$ the set $Y_v\in H^*_\red$ containing $v$ and the centre of 
the star to which $v$ belongs.
With this notation,~\eqref{eq:Asia's-trick} rewrites as 
\begin{equation}\label{eq:55}
	\text{if } f\in H \text{ and } v\in f\cap V_S, \text{ then }
	Y_v\subseteq f\,.
\end{equation}

Condition~\ref{it:VSS} of Definition~\ref{dfn:family} is an immediate consequence of 
this statement and it also follows that $H_S\supseteq \{h\in H\colon |h\cap V_S|=rt\}$.
The reverse inclusion is implied by~\eqref{eq:VS} and thereby condition~\ref{it:VS}
is proved as well.

It remains to establish~\ref{it:VT}, i.e., that for  
\[
	V_R=V\setminus (V_T\cup V_S)
	\qquad \text{ and } \quad 
	H_R= H\setminus (H_T\cup H_{S})
\]
we have 
\begin{equation}\label{eq:HR-goal}
	|H_R|\le |V_T||V_S|n^{rt-3}+(rt^2)^{r^3t^6} |V_R|n^{rt-2}\,.
\end{equation}

The first step in the proof of this result is to split $H_R$ into the two 
subhypergraphs $H_{ST}$ and $\widehat{H}_R$ with the intention of proving 
$|H_{ST}|\le |V_T||V_S|n^{rt-3}$ and 
$|\widehat{H}_R|\le (rt^2)^{r^3t^6} |V_R|n^{rt-2}$.

The family $H_{ST}$ is defined by
\[
	H_{ST}=\bigl\{h\in H_R\colon \text{there are } v\in h\cap V_S 
	\text{ and a team } g\subseteq h
	\text{ with } Y_v\cap g=\varnothing\bigr\}\,.
\]

Observe that if $h\in H_{ST}$ and $v$, $g$ are as in the above definition,
then $Y_v\subseteq h$ follows from~\eqref{eq:55} and we have 
$|h\sm(Y_v\cup g)|=rt^2-1-t-rt(t-1)\le rt-3$. As there are at most 
$|V_S|$ possibilities for $v$, $|V_T|$ possibilities for $g$, 
and $n^{rt-3}$ possibilities for the set $h\sm(Y_v\cup g)$, it follows that
we have indeed 
\[
	|H_{ST}|\le |V_T||V_S|n^{rt-3}\,.
\]

Thus to conclude the argument we need to show that the hypergraph 
$\widehat{H}_R=H_R\setminus H_{ST}$ satisfies 
\begin{equation}\label{eq:HR-hut}
	|\widehat{H}_R|\le(rt^2)^{r^3t^6} |V_R|n^{rt-2}\,.
\end{equation}

In the special case $V_R=\varnothing$ this can only be true if $\widehat{H}_R=\varnothing$
holds as well. For that reason it will certainly help us to establish
\begin{equation}\label{eq:needed?}
	H_R\setminus V_R\subseteq H_{ST}\,.
\end{equation}
To verify this, consider any edge $f\in H_R$ not meeting $V_R$. 
Owing to $f\not\in H_T$ there must exist a vertex $v\in f\cap V_S$ 
and~\eqref{eq:55} tells us that $Y_v\subseteq f$. 
Now $f\sm Y_v$ cannot be a subset of~$V_S$ because~\eqref{eq:55} would then yield $f\in H_S$.
Together with $f\subseteq (V_S\cup V_T)$ this shows that there must be a vertex
$x\in V_T\cap (f\sm Y_v)$. This vertex must in turn belong to some team~$g\in G$,
which is in fact a subset of $f\sm Y_v$. Thereby~\eqref{eq:needed?} is proved. 

Due to the discussion preceeding~\eqref{eq:needed?} we may henceforth suppose that 
$V_R\ne\varnothing$. Now the idea for proving~\eqref{eq:HR-hut} is that we can mark 
in every edge $h\in \widehat{H}_R$ at least one vertex from $h\cap V_R$ in such a way 
that every vertex in $V_R$ gets marked 
at most~$(rt^2)^{r^3t^6}n^{rt-2}$ many times. The marking procedure we use depends on 
the red and green sets contained in~$h$ and thus it involves several case distinctions. 

In view of property~\ref{dl:dc3} of the green sets, we may write 
\begin{equation} \label{eq:HR-123}
	\widehat{H}_R=\widehat{H}^1_R\cup \widehat{H}^2_R\cup \widehat{H}^3_R
\end{equation} 
with
\begin{align*}
	\widehat{H}^1_R &=\bigl\{h\in \widehat{H}_R\colon h 
		\text{ cannot be written as a union of red and green sets}\bigr\}\,, \\
	\widehat{H}^2_R &=\bigl\{h\in \widehat{H}_R\colon h 
		\text{ is the union of its green subsets}\bigr\}\,, \\
	\text{ and } \quad \widehat{H}^3_R &=\bigl\{h\in \widehat{H}_R\colon 
		h\not\in \widehat{H}^1_R \text{ and there is some } f\in H^*_\red 
		\text{ with } f\subseteq h\bigr\}\,. 
\end{align*}

Regarding the first of these three hypergraphs, we note that if $h\in \widehat{H}^1_R$
and $v\in h$ is not contained in any red or green subset of $h$, then $v\in V_T$ is impossible
due to the inseparability of the teams, $v\in V_S$ is impossible by~\eqref{eq:55},
and hence we must have $v\in V_R$. In other words, if we set
\[
	H_v=\bigl\{h\in\widehat{H}_R\colon v\in h\bigr\}
\]
and 
\[
	\widehat{H}^1(v) =\bigl\{h\in H_v\colon \text{there is no } f\in H_\red \text{ with } 
				v\in f\subseteq h\bigr\}
\]
for every $v\in V_R$, then
\begin{equation}\label{eq:H1}
	\widehat{H}^1_R\subseteq\bigcup_{v\in V_R}\widehat{H}^1(v)\,.
\end{equation}

According to our plan the hypergraphs $\widehat{H}^1(v)$ should be of size at most 
$O(n^{rt-2})$ and this is indeed what we prove next. 
 
\begin{fact}\label{fact:h1}
	For every $v\in V_R$ we have 
		\[
		\big|\widehat{H}^1(v)\big|\le (rt^2)^{2rt^2}n^{rt-2}\,.
	\]
	\end{fact}

\begin{proof}
	Assume for the sake of contradiction that $v\in V_R$ violates this claim.
	Then $x=\{v\}$ is an example of a subset of $V$ with $v\in x$ and 
		\begin{equation}\label{eq:h1-x}
		\big|\bigl\{h\in \widehat{H}^1(v)\colon x\subseteq h\bigr\}\big|
		> (rt^2)^{2(rt^2-|x|)}n^{rt-2}\,.
	\end{equation}
	Now let $x\subseteq V$ be a maximal set of vertices with $v\in x$ that 
	satisfies~\eqref{eq:h1-x}. As $x\subseteq h$ for some $h\in H$, 
	we must have $|x|\le rt^2$. Thus 
		\[
		n^{rt-2}\le (rt^2)^{2(rt^2-|x|)}n^{rt-2} <
	  	\big|\bigl\{h\in \widehat{H}^1(v)\colon x\subseteq h\bigr\}\big|
	  	\le \binom{n-|x|}{rt^2-|x|}\le n^{rt^2-|x|}
	\]
		and it follows that $|x|\le rt(t-1)+1$. But owing to the definition of 
	$\widehat{H}^1(v)$ it is not possible for $x$ to be red. This means, in particular, that
	there is a maximal $\Delta$-system~$\cG\subseteq \widehat{H}^1(v)$ with kernel~$x$
	and $|\mathcal{G}|< rt^2$. The size of the set
	$B=\bigcup_{h\in \mathcal{G}}(h\setminus x)$ can be bounded by 
	$|B|\le \sum_{h\in \mathcal{G}}|h|< (rt^2)^2$ and
	the maximality of $\cG$ implies that every edge $h\in \widehat{H}^1(v)$ 
	with~$x\subseteq h$ intersects $B$. So by averaging and~\eqref{eq:h1-x} there 
	exists a vertex $w\in B$ with
	\[
		\big|\bigl\{h\in \widehat{H}^1(v)\colon (x\cup\{w\})\subseteq h\bigr\}\big|
		> \frac{(rt^2)^{2(rt^2-|x|)}}{(rt^2)^2}n^{rt-2}
		= (rt^2)^{2(rt^2-|x\cup\{w\}|)} n^{rt-2}\,.
	\]
		
	This inequality tells us that $x\cup\{w\}$ contradicts the maximality of $x$. 
	Thereby Fact~\ref{fact:h1} is proved.
\end{proof}

This completes our analysis of $\widehat{H}^1_R$ and we proceed with $\widehat{H}^2_R$.
To this end, we shall use the trivial decomposition  
\begin{equation}\label{eq:H2}
	\widehat{H}^2_R=\bigcup_{v\in V_R}\widehat{H}^2(v)\,,
\end{equation}
where
\[
	\widehat{H}^2(v)=\bigl\{h\in H_v: \text{$h$ is the union of its green subsets}\bigr\}\,.	
\]

\begin{fact}\label{fact:h2}
	If $v\in V_R$, then
		\[
		\big|\widehat{H}^2(v)\big|< (rt^2)^{r^3t^6-7rt^2}n^{rt-2}\,.
	\]
	\end{fact}

\begin{proof}
	Consider the auxiliary set system 
		\[
		\cG=\bigl\{x\subseteq V\colon 2\le |x|\le rt^2 \text{ and } 
		G|x \text{ is connected}\bigr\}\,. 
	\]
		
	Utilising property~\ref{dl:dc1} of the green sets and the fact that for every
	$x\in\cG$ there is a spanning sub-setsystem of $G|x$ consisting of at most $|x|$
	sets we obtain 
	\[
		\Delta (\cG)<(\Delta(G)rt^2)^{rt^2}
		\le (rt^2)^{2(rt^2)^2+rt^2}\,.
	\]
	
	[Why? Fix $v\in V$ and look at an arbitrary edge $x\in\G$ with $v\in x$. Due 
	to the connectedness of $G|x$ there exist $g_1, \ldots, g_\ell\in G|x$ with $v\in g_1$, 
	$(g_1\cup \ldots\cup g_{i-1})\cap g_i\ne \varnothing$ for~$i\in [2, \ell]$, and 
	$g_1\cup\ldots\cup g_\ell=x$.
	There are at most $rt^2$ possibilities for $\ell$, $\Delta(G)$ possibilities for $g_1$,
	and for every $i\in [2, \ell]$ there are at most $|g_1\cup \ldots\cup g_{i-1}|\Delta(G)$
	possibilities for $g_i$, which is at most $\Delta(G)rt^2$.]
	As every edge of $G$ has at least $t$ vertices, the same is true about $\cG$. 
	Moreover, the only possibility for $x\in\cG$ to have size exactly $t$ is that
	it is a team. 
	 
	Now any given $h\in \widehat{H}^2(v)$ can be expressed as a disjoint union of 
	edges of $\cG$ by looking at the connected components of $G|h$. 
	The number of edges from $\cG$ appearing in such a decomposition can be at most $rt-1$
	because of the remarks from the previous paragraph and as $v$ cannot belong to a team.
	
	Representing each edge $h\in \widehat{H}^2(v)$ by a selection of one vertex from
	each of its at most~$rt-2$ green components not containing $v$, we learn that indeed  
		\[
		\big|\widehat{H}^2(v)\big|\le \Delta(\cG)^{rt-1}n^{rt-2}
		<(rt^2)^{r^3t^6-7rt^2}n^{rt-2}\,. \qedhere
	\]
	\end{proof}
	
	It remains to deal with the hypergraph $\widehat{H}^3_R$, which may be further decomposed 
	as
		\begin{equation}\label{eq:H3}
	\widehat{H}^3_R=\widehat{H}^{3, \mathrm{x}}_R\cup \widehat{H}^{3, \mathrm{y}}_R\,,
	\end{equation}  	
		where 
		\[
	\widehat{H}^{3, \mathrm{x}}_R=\bigl\{h\in \widehat{H}^3_R\colon
	h \text{ is the union of its subsets belonging to }H^*_\red \bigr\}
	\]
		and $\widehat{H}^{3, \mathrm{y}}_R=\widehat{H}^3_R\sm \widehat{H}^{3, \mathrm{x}}_R$.
	We will estimate the sizes of these two hypergraphs in the two facts that follow.
	In both proofs we will frequently use the inequality~$|H^*_{\red}|\le n$ obtained 
	in~\eqref{eq:h7} above without referencing it.
	
	\begin{fact}\label{fact:h3}
		We have $\big|\widehat{H}^{3, \mathrm{x}}_R\big|\le (rt^2)^{rt^2}|V_R|n^{rt-2}$.
	\end{fact}
	
	\begin{proof}
		In the light of~\eqref{eq:rh} there are only two possibilities for an edge $h\in \widehat{H}^{3, \mathrm{x}}_R$. Either
		\begin{enumerate}[label=\rmlabel]
			\item\label{it:x1} there are $f,f'\subseteq h$ in $H^*_\red$ such that 
			$|f\cap f'|\le rt(t-1)-t$,
			\item\label{it:x2} or there is some $Y_h\subseteq h$ of size $rt(t-1)$ such that 
			$Y_h\cup\{v\}\in H^*_\red$ holds for every vertex $v\in h\sm Y_h$.
		\end{enumerate}
		
		If $h$ is of type~\ref{it:x1} we have $|f\cup f'|=|f|+|f'|-|f\cap f'|\ge rt^2-rt+t+2$ 
		and hence $|h-(f\cup f')|\le rt-t-2\le rt-4$. As there are at most $n^2$
		possibilities to choose a pair $f, f'$ of two edges from $H^*_\red$ and at most 
		$n^{rt-4}$ possibilities to choose at most $rt-4$ further vertices in $V$, 
		there can be at most $n^{rt-2}$ edges in $\widehat{H}^{3, \mathrm{x}}_R$ 
		to which the description~\ref{it:x1} applies.
		
		Next we note that if $h$ and $Y_h$ are as in~\ref{it:x2}, then $Y_h$ cannot be 
		purple for otherwise $h$ would satisfy the requirements for belonging to $H_S$.
		We will prove below that there are at most $9|V_R|$ such edges in the 
		special case $r=1$ and $t=2$, and at most $\binom{rt^2}{rt}\cdot rt^2n$
		such edges if $rt\ge 3$. Due to $V_R\ne \varnothing$ this suffices to establish 
		Fact~\ref{fact:h3} in both cases.
		
		Let us consider the case that $r=1$ and $t=2$ first. If $h$ denotes an edge 
		of type~\ref{it:x2}, then $h\sm Y_h\subseteq V_R$ by~\eqref{eq:55} and~\ref{dl:dc4},
		and we may mark any vertex $v\in h\sm Y_h$. 
		Since the triples $Y_h\cup\{v\}$ and $h\sm \{v\}$ are both in 
		$H^*_\red$,~\eqref{eq:rh} implies that $v$ is contained in at most~$3$ red 
		sets $f\in H^*_\red$. For none of them $f\setminus \{v\}$ is purple
		(because $v\not\in V_S$), which in turn means that each of them 
		can be involved at most 3 times in the marking of $v$. 
		Altogether each~$v\in V_R$ gets marked at most $9$ times due to edges of 
		type~\ref{it:x2}, wherefore there are indeed at most~$9|V_R|$ such edges.
		
		Now suppose that $rt\ge 3$ and let $h$ again denote an edge of type~\ref{it:x2}. 
		As $Y_h$ arises from a member of $H^*_\red$ by the deletion of a vertex, 
		there are at most $rt^2n$ candidates for this set and each of them can be used 
		in at most $\binom{rt^2}{rt}$ edges of type~\ref{it:x2}, for otherwise it would 
		be purple. This proves the upper bound of $\binom{rt^2}{rt}\cdot rt^2n$
		on the number of such edges $h$ and the proof of Fact~\ref{fact:h3}
		is complete.
\end{proof}
	
\begin{fact}\label{fact:h4}
	We have $\big|\widehat{H}^{3, \mathrm{y}}_R\big|\le \bigl((rt^2)^{2rt^2} +1\bigr)|V_R| n^{rt-2}$.
\end{fact}

\begin{proof}
	Consider any edge $h\in \widehat{H}^{3, \mathrm{y}}_R$. Since $h\in \widehat{H}^3_R$,
	there is a set $f\in H^*_\red$ with $f\subseteq h$. If $h\cap V_S\ne\varnothing$ 
	we may suppose by~\eqref{eq:55} that $f=Y_u$ holds for some $u\in h\cap V_S$.
	
	By $h\not\in \widehat{H}^{3, \mathrm{x}}_R$ there exists a vertex $v\in h\sm f$
	that is not contained in any member of $H^*_\red$ which at the same 
	time happens to be a subset of $h$. Therefore $h\not\in \widehat{H}^1_R$ tells us
	that there exists a set $\widehat{g}\in \widehat{H}_\red\cup G$ 
	with $v\in \widehat{g}\subseteq h$. Due to property~\ref{dl:dc3} of $G$ this leads us 
	to a green set~$g$ with $v\in g\subseteq h$.
	Because of~\ref{dl:dc4} the numbers $|g|$ and $|f\cap g|$ are divisible by $t$,
	and hence so is $|g\sm f|$. Thus it follows from $v\in g\sm f$ that $|g\sm f|\ge t$,
	wherefore
	\begin{equation}\label{eq:hfg}
		|h\sm (f\cup g)|\le rt-t-1\le rt-3\,. 
	\end{equation}
		
    By~\eqref{eq:55} and the choice of $v$ we have $v\not\in V_S$ and, hence, $v$ is 
    either in $V_R$ or in $V_T$. Let us analyse these two possibilities separately. 
    
	If $v\in V_R$, then we mark it. Property \ref{dl:dc1} of $G$ tells us that $v$ is 
	contained in at most $(rt^2)^{2rt^2}$ green sets and, using~\eqref{eq:hfg}, 
	one can conclude that in this way each vertex of $V_R$ is marked at 
	most $(rt^2)^{2rt^2}n^{rt-2}$ many times.
	
	On the other hand, if $v\in V_T$, then $g$ is a team. Since $|g\sm f|\ge t$, the sets 
	$f$ and $g$ are disjoint. By $h\not\in H_{ST}$ it follows that $f$ is not of the 
	form $Y_u$ with $u\in h\cap V_S$, and by our choice of $f$ this yields 
	$h\cap V_S=\varnothing$.
	Moreover $|h\sm f|\equiv -1\pmod{t}$ and, therefore, it is not possible that~$h\sm f$ is 
	entirely covered by teams. Consequently there is a vertex 
	$w\in \bigl(h\sm (f\cup g)\bigr)\cap V_R$ that can be marked. Now
	there are at most $n$ possibilities for $f$, for $g$, and for each of 
	the remaining vertices in $h\sm (g\cup f\cup \{w\})$. 
	Using~\eqref{eq:hfg} again, we get that in this way each vertex is marked 
	at most $n^{rt-2}$ further times. 
	
	Summarising the above estimations one obtains  
		\[
		\big|\widehat{H}^{3, \mathrm{y}}_R\big|
		\le  |V_R|(rt^2)^{2rt^2}n^{rt-2} +|V_R|n^{rt-2}\,.\qedhere
	\]
	\end{proof}

Collecting all the above results we get
\[
	\big|\widehat{H}^1_R\big|\le (rt^2)^{2rt^2}|V_R|n^{rt-2}
\]
from~\eqref{eq:H1} and Fact~\ref{fact:h1},
\[
	\big|\widehat{H}^2_R\big|\le (rt^2)^{r^3t^6-7rt^2}|V_R|n^{rt-2}
\]
from~\eqref{eq:H2} and Fact~\ref{fact:h2},
\[
	\big|\widehat{H}^3_R\big|\le 2(rt^2)^{2rt^2}|V_R|n^{rt-2}
\]
from~\eqref{eq:H3}, Fact~\ref{fact:h3}, Fact~\ref{fact:h4}, 
and finally 
\[
	|\widehat{H}_R|\le \bigl(3(rt^2)^{2rt^2}+(rt^2)^{r^3t^6-7rt^2}\bigr)|V_R|n^{rt-2}
	\le (rt^2)^{r^3t^6}|V_R|n^{rt-2}
\]
from~\eqref{eq:HR-123} and the three previous estimates. 
This concludes the proof of~\eqref{eq:HR-hut} and, hence, the proof of the 
Structure Theorem \ref{thm:dt}.

\end{document}